\documentclass[final,3p,times]{elsarticle}

\usepackage{amsmath,amssymb,amsfonts}
\usepackage{amsthm}
\usepackage{hyperref}
\usepackage{cleveref}

\usepackage{mathrsfs}
\usepackage{caption}
\usepackage{graphicx}
\usepackage{psfrag}
\usepackage{epstopdf}
\usepackage{tikz}
\usetikzlibrary{automata}
\usetikzlibrary{arrows,decorations.pathmorphing,backgrounds,positioning,fit}
\usetikzlibrary{shapes.geometric}
\usetikzlibrary{circuits.ee,circuits.ee.IEC}
\usepackage{subfigure}
\usepackage{tikz-3dplot}
\usepackage{pgfplots}

\newcommand{\bm}[1]{{\boldsymbol{#1}}}

\theoremstyle{definition}

\newtheorem{example}{Example}
\newtheorem*{invprob}{Inverse Problem}

\theoremstyle{plain}
\newtheorem{theorem}{Theorem}
\newtheorem{proposition}{Proposition}

\newtheorem{corollary}{Corollary}

\theoremstyle{remark}

\newtheorem{remark}{Remark}

\usepackage{lineno}

\begin{document}

\begin{frontmatter}



\title{On Lyapunov functions and gradient flow structures in linear consensus systems}


\author[hma]{Herbert Mangesius}
\address[hma]{Department of Electrical, Electronic and Computer Engineering, Technische Universit\"{a}t M\"{u}nchen, Arcisstrasse 21, D-80209 Munich, Germany}
\fntext[]{The corresponding author is H. Mangesius.
The work was accomplished during the first author's research stay at ICTEAM, Universit\'{e} catholique de Louvain.}
\ead{mangesius@tum.de}

\author[jcd]{Jean-Charles Delvenne}
\address[jcd]{ICTEAM and CORE, Universit\'{e} catholique de Louvain, 4 Avenue Lema\^{i}tre, B-1348 Louvain-la-Neuve, Belgium}

\begin{abstract}
Quadratic Lyapunov functions are prevalent in stability analysis of linear consensus systems. In this paper we show that weighted sums of convex functions of the different coordinates are Lyapunov functions for irreducible consensus systems.
We introduce a particular class of additive convex Lyapunov functions that includes as particular instances the stored electric energy in RC circuits, Kullback-Leiber's information divergence between two probability measures, and Gibbs free energy in chemical reaction networks. 
On that basis we establish a general gradient formalism that allows to represent linear symmetric  consensus dynamics as a gradient descent dynamics of any additive convex Lyapunov function, for a suitable choice of local scalar product.
This result generalizes the well-known Euclidean gradient structures of consensus 
to general Riemannian ones. 
We also find natural non-linear consensus dynamics differing only from linearity by a different Riemannian structure, sharing a same Lyapunov function. We also see how information-theoretic Lyapunov functions, common in Markov chain theory, generate linear consensus through an appropriate gradient descent.  From this unified perspective we hope to open avenues for the study of convergence for linear and non-linear consensus alike.
%
\end{abstract}


\end{frontmatter}


\section{Introduction}
The  linear consensus algorithm, cf., e.g., \cite{Blondel2005} and \cite{Murray2007} for
 an overview and further references, is a generic element in network systems theory, where it appears as an essential building block at the heart of diverse network applications in system modelling, design, but also in methods for system analysis.  
In proving stability, quadratic Lyapunov functions are a prominent choice, that additionally provide appealing interpretations as vanishing disagreement within a network of interacting agents.
 As noted in \cite{Moreau2004}, the particular type of functions that yield a basis to construct possible Lyapunov functions,
 ``may point towards generalizations of the present theory [of linear consensus systems] to more complex coordination models, possibly involving non-linear [...] dynamics''.
In this context, note that quadratic Lyapunov functions underlie the particular hypothesis that (internal) dissipation processes are linear, cf., e.g., \cite{Willems1972p2}, 
thus making it interesting which types of functions, other than quadratic ones, do also yield Lyapunov inequalities.
Gradient approaches to stability analysis play a particular role, as a Lyapunov function that is also a potential for the system (in the sense that the dynamics can be written as a   gradient descent of the Lyapunov function) determines, besides the qualitative (asymptotic) behaviour, also the quantitative behaviour.
Here, we are interested in characterizing the set of functions that are Lyapunov functions for irreducible consensus systems, under the hypothesis that Lyapunov function candidates are additive - a natural requirement for network systems, cf., e.g., \cite{Willems1972p1}.
For symmetric consensus systems we seek to characterize conditions that render an additive Lyapunov function a potential, in which the linear dynamics generates a gradient flow. 

Related to this problem is the work \cite{Willems1976}, where J.C. Willems shows that convexity is a sufficient condition for an additive function to be a Lyapunov function for a diagonally dominant linear system.
The commonly known gradient structure for symmetric consensus systems in the Laplacian potential, also-called the group disagreement, is for instance explained in \cite{Murray2004}. In the work \cite{vdSchaft2011a}, A. van der Schaft 
shows that the gradient dynamics in the group disagreement can equivalently be seen as
gradient dynamics in the well-known sum-of-squares Lyapunov function, also called the collective disagreement, see \cite{Murray2007}. 

The main contributions of this work are as follows:
 We give a necessary and sufficient characterization of a weighted additive function to be a (strict) Lyapunov function for irreducible consensus systems.
The Euclidean gradient flow structure with potential the collective disagreement is generalized to Riemannian gradient dynamics in additive strictly convex potentials.
We provide a novel output feedback representation of LTI symmetric consensus dynamics, that coincides with a non-linear port-Hamiltonian realization, in which a non-linear controller, given by a state dependent Laplace matrix, interacts with a non-linear integrator system. 
The disagreement measures introduced in \cite{Murray2004} and  \cite{Murray2007} are generalized to disagreement functions in which convexity is instrumental.
We provide a non-linear circuits context that seems to be novel to linear consensus theory, where the Lyapunov function of a consensus system is interpreted as the energy stored
in a nonlinear capacitance, while the Riemannian geometry of the state space is encoded into the behaviour of non-linear resistances. For a suitable choice of geometry, these non-linearities cancel out and create a linear consensus dynamics for the voltages at different nodes of the circuit. If we decouple the geometry from the energy function, we are able to explore a variety of non-linear consensus dynamics, as we exemplify on a thermodynamic example. 
    

\section{Preliminaries \label{sec:prel}}

\subsection{Gradient systems \label{ssec:gradsys}}
Gradient descent systems, or gradient flows, are dynamical systems whose trajectories follow the steepest descent of a scalar potential function $V$. In Euclidean space with the usual scalar product this is expressed by the familiar equation
\begin{equation}
\dot{\bm{x}} = -\nabla V(\bm{x}),
\label{eq:usualgrad}
\end{equation} 
where $\bm{x} \in M \subset \mathbb{R}^n$. 
In state coordinates where the scalar product is described by a positive definite and symmetric matrix $\textbf{G}$ as $\langle \bm{x}, \bm{y} \rangle= \bm{x}\cdot \textbf{G}\bm{y}$, this equation takes the more general form
\begin{equation}
\dot{\bm{x}} = -\textbf{G}^{-1}\nabla V(\bm{x}).
\label{eq:Ggrad}
\end{equation}
This is swiftly generalized to any Riemannian manifold, which is a smooth manifold where the tangent space $T_{\bm{x}}M$ at every point $\bm{x}\in M$ is endowed with a scalar product, that varies smoothly in $\bm{x}$.
In some system of coordinates for the manifold, this scalar product is determined by a variable, positive definite and symmetric matrix $\textbf{G}(\bm{x})$, and gradient flows are naturally represented by 
\begin{equation}
\dot{\bm{x}} = -\textbf{G}^{-1}(\bm{x})\nabla V(\bm{x}),
\label{eq:GgradRiem}
\end{equation}
see for instance \cite{SimpsonPorco2014}.

 In this paper we are interested in the following general inverse problem:
\begin{invprob}
 Given differentiable dynamics $\dot{\bm{x}}= \bm{f}(\bm{x})$ 
  with Lyapunov function $V$, find a Riemannian structure defined by $\textbf{G}^{-1}(\cdot)$ for which the dynamics is a gradient descent of $V$,  i.e.,  \mbox{$\bm{f}(\bm{x})=-\textbf{G}^{-1}(\bm{x})\nabla V(\bm{x})$}.
\end{invprob}

As we shall see, the state equations of consensus systems do not uniquely characterize a gradient structure, so that the Inverse Problem has several solutions in general. We find a specific solution that respects the network topology, i.e.,  $\textbf{G}^{-1}$ and $\textbf{L}$ share the same zero/nonzero pattern.

\subsection{Consensus systems} 
A linear time-invariant (LTI) consensus system
is defined by the dynamics in vector  and component form
\begin{equation}\label{eq:consdyn}
\dot{\bm{x}}(t)=-\textbf{L}\bm{x}(t), \ \ \ 
\dot{x}_i(t)=\sum_{j\not = i}^n l_{ij}(x_j(t) - x_i(t)), \ \ 1\leq i\leq n,
\end{equation}
where $\textbf{L}$ is a positive semi-definite graph Laplace matrix.
Let $\mathsf{G}=(N,E,w)$ be a weighted directed graph with
 $N=\{1,\ldots,n\}$ the set of nodes, $E \subseteq N \times N$, the set of edges, with elements ordered pairs $(i,j)\in E$, and with weighting function $w:E\to \mathbb{R}_{> 0}$ that assigns to each directed edge~$(i,j) \in E$ a coupling strength given by a positive number $w_{ij}$.
Then, the graph Laplace matrix is defined as
\begin{equation}\label{eq:graphlaplacianmatw}
 [\textbf{L}]_{ij}=l_{ij}:= \left\{\begin{tabular}{l l}
                                      $-w_{ij}$, & if \ \ $(i,j) \in E$, \\
                                      $0$, & if \ \ $(i,j) \not\in E$, \\
                                      $\sum_{j: (i,j) \in E} w_{ij}$, & if \ \ $j=i \in N$.
                                     \end{tabular}\right. 
\end{equation}

The dynamical system generated by dynamics \eqref{eq:consdyn}
 has flow map 
 \begin{equation}\label{eq:perronmatrix}
\textbf{P}: \ \ \ \textbf{P}^{(t)}=e^{-\textbf{L}t},  \ \ \ 0 \leq t,
\end{equation}
which is a stochastic matrix, \cite{Willems1976,Murray2007,Lawson1999}, i.e., satisfying
\begin{equation}\label{eq:perronmatrix1}
[\textbf{P}^{(t)}]_{ij} \geq 0,  \ \ \ \sum_{j=1}^n [\textbf{P}^{(t)}]_{ij}=1, \ \ \forall i\in N,\ \ \text{and}\ \ t \geq 0.
\end{equation}
Eigenvectors of $\textbf{L}$ and $\textbf{P}^{(t)}$ for any $t>0$ are identical, cf., e.g., \cite{Murray2007}. 
In particular, the vector of ones, $\bm{1}$, occurs to be a right eigenvector 
associated to $\lambda_1(\textbf{L})=0$, the eigenvalue of smallest modulus of $\textbf{L}$, respectively of $\lambda_1(\textbf{P})=1$, the eigenvalue of maximum modulus of the transfer matrix.
Let $\bm{q} \in \mathbb{R}^n_{\geq 0}$ denote the associated left eigenvector.
Then,
\begin{equation}\label{eq:specpropPL}
\bm{q}^\top\textbf{L}=\bm{0}^\top \ \text{and} \ \ \textbf{L}\bm{1}=\bm{0}
\ \ \Leftrightarrow \ \ \bm{q}^\top\textbf{P}^{(t)}=\bm{q}^\top \ \text{and} \ \ \textbf{P}^{(t)}\bm{1}=\bm{1}.
\end{equation}
We normalize $\bm{q}$ such that $||\bm{q}||_1=1$; then, $\bm{q}$ is called Perron vector of $\textbf{P}^{(t)}$.

A graph $\mathsf{G}$ is said to be strongly connected if $\mathsf{rank}(\textbf{L})=n-1$.
In this case,  $\textbf{P}^{(t)}$ for any $t>0$ is an irreducible matrix, so that from
Perron-Frobenius theory for non-negative matrices, cf., e.g., \cite{Meyer2001} ch. 8, it is known that $q_i>0$, for all $1\leq i\leq n$.
A consensus system is said to be average preserving if in addition to row sums also column elements of the Laplacian matrix (the transfer matrix) sum to zero (one).
In this case the Perron vector is given as $\bm{q}=\frac{1}{n}\bm{1}$.
A consensus system is said to be symmetric if $l_{ij}=l_{ji}$ for all $i,j\in N$.

From  $\mathsf{ker}(\textbf{L})=\{c\bm{1},c\in \mathbb{R}\}$ it follows
translation invariance of dynamics by constant vectors, i.e., $\textbf{L}\bm{x}=\textbf{L}(\bm{x}+c\bm{1})$, for any $c\in \mathbb{R}$. 
Due to $\bm{q}^\top\textbf{P}^{(t)}=\bm{q}^\top$, the weighted average
\begin{equation}\label{eq:invquant}
\sum_{i\in N}q_ix_i(t)=\bm{q}^\top\textbf{P}^{(t)}\bm{x}(0)=\bm{q}^\top\bm{x}(0)=constant, \ \ \forall t\geq 0,
\end{equation}
i.e., it remains invariant along trajectories of a consensus system, so that
 an initialization of \eqref{eq:consdyn} with $\bm{x}(t=0)=\bm{x}_0$ 
defines the asymptotically reached agreement value
$\mathsf{a}(\bm{x}_0):=\bm{q}^\top\bm{x}_0$,
i.e., $\lim_{t\to\infty} \textbf{P}^{(t)}\bm{x}_0=
\frac{\bm{1}\bm{q}^\top}{\bm{q}^\top\bm{1}}\bm{x}_0=\bm{q}^\top\bm{x}_0\bm{1}=\mathsf{a}(\bm{x}_0)\bm{1}$.

Due to these invariance properties, without loss of generality, we may set
 as configuration space for irreducible consensus systems
the $n-1$ dimensional simplex
\begin{equation}\label{eq:hyperplane}
M(\alpha):=\left\{ \bm{x} \in \mathbb{R}_{> 0}^n:  \sum_{i=1}^n q_ix_i= \alpha=constant \right\}, \ \ \ \alpha >0,
\end{equation}
which contains all positive trajectories of final consensus value $\alpha$.
In \eqref{eq:hyperplane}, the parameter $\alpha$ can be chosen freely from $\mathbb{R}_{>0}$, because the LTI dynamics is homogeneous,
i.e., $|c \textbf{L}\bm{x}|=c|\textbf{L}\bm{x}|, c\in \mathbb{R}_{>0}$,  
so that (dynamical) system properties are the same for every value $\alpha$.
The peculiarities of a specific application may however motivate
to choose a particular value for $\alpha$. The motivation for focusing on the simplex as a configuration space comes from the following example.

\begin{example}[Markov chains \& consensus system configurations]\label{ex:configspace}
Choosing for instance a scaling constant $c= \mathsf{a}(\bm{x}_0)^{-1}/n$, yields for balanced consensus systems the configuration space~\mbox{$M(\alpha=1/n)=\{ \bm{x} \in \mathbb{R}_{> 0}^n: \sum_{i\in N}x_i=1 \}$}, which is the set of positive discrete probability vectors; the associated equilibrium state $\frac{1}{n}\bm{1}$ then corresponds to the invariant probability distribution, cf., e.g., \cite{Cover1991}.
\end{example}

\subsection{Three Lyapunov functions \&  Disagreement   \label{ssec:3liap}}

%

The Laplacian potential is defined as
\begin{equation}\label{eq:laplacianpotentialr}
V_{\textnormal{L}}(\bm{x}):=\frac{1}{2}\bm{x}\cdot \textbf{L}\bm{x}=\frac{1}{2}\sum_{(i,j) \in E} l_{ij} (x_i-x_j)^2,
\end{equation}
and it specifies the \textit{group disagreement} $\Psi_{\textnormal{L}}(\bm{x}):=2V_{\textnormal{L}}(\bm{x})$. Both functions are well known Lyapunov functions for symmetric consensus systems~\cite{Murray2004,Murray2007}.
With  \eqref{eq:laplacianpotentialr} as Lyapunov function, the symmetric dynamics \eqref{eq:consdyn} can be written as
\begin{equation}\label{eq:laplgradflow}
\dot{\bm{x}}=-\nabla V_{\textnormal{L}}(\bm{x}),
\end{equation}
in which we recognize a gradient flow on the Euclidean space $\mathbb{R}^n$ with the canonical scalar product, where $\textbf{G}=\textbf{I}=\textbf{G}^{-1}$.

Another common Lyapunov function for average preserving systems is simply the square distance to consensus, for the canonical two-norm, leading to the sum-of-squares function
\begin{equation}\label{eq:Vsos}
V_{\textnormal{SoS}}(\bm{x}):=
\frac{1}{2}\sum_{i=1}^n\left(x_i-\mathsf{a}(\bm{x})\right)^2=\frac{1}{2}||\bm{x}-\mathsf{a}(\bm{x})\bm{1}||^2.
\end{equation}
which is also called the \textit{collective disagreement} \cite{Murray2007}. 
If the consensus dynamics is symmetric, then $\dot{\bm{x}} = - \textbf{L}\bm{x}$ does not represent a gradient flow of \eqref{eq:Vsos} for the usual canonical scalar product. However, it is a gradient dynamics for the
 scalar product determined by $\textbf{G}^{-1}=\textbf{L}$:
With $\nabla V_{\textnormal{SoS}}(\bm{x})=\bm{x}-\mathsf{a}(\bm{x})\bm{1}$ we can write 
\begin{equation}
\dot{\bm{x}}=-\textbf{G}^{-1}\nabla V_{\textnormal{SoS}}(\bm{x})=-\textbf{L}(\bm{x}-\mathsf{a}(\bm{x})\bm{1}) =-\textbf{L}\bm{x},
\end{equation}
as $\mathsf{a}(\bm{x})\bm{1} \in \mathsf{ker}(\textbf{L})$.
This is an instance
where $\textbf{G}^{-1}=\textbf{L}$ is not positive definite, but semi-definite, yet the matrix determines a positive definite inner product $\nabla V_{\textnormal{SoS}}(\bm{x})\cdot \textbf{L} \nabla V_{\textnormal{SoS}}(\bm{x})$ for all $\bm{x}\in M$,
because $\nabla V_{\textnormal{SoS}}(\bm{x}) \in \mathsf{ker}(\textbf{L})$ if and only if $\bm{x}\in \mathsf{ker}(\textbf{L})$, i.e., if and only if the consensus system has reached the equilibrium point.

%

The sum-of-squares potential is in close relation with the Laplacian potential and group disagreement, as for symmetric consensus systems the latter occurs to be the dissipation rate of the former: 
Differentiating w.r.t. time yields
\begin{equation}\label{eq:Vsoscontraction}
\dot{V}_{\textnormal{SoS}}(\bm{x})\equiv -\nabla V_{\textnormal{SoS}}(\bm{x}) \cdot \textbf{L}\nabla V_{\textnormal{SoS}}(\bm{x})= - ||\nabla V_{\textnormal{SoS}}(\bm{x})||_{\textbf{L}}^2=-\Psi_{\textnormal{L}}(\bm{x}).
\end{equation}

A third Lyapunov function arises as an interpretation of the average preserving consensus equation as the evolution of a probability distribution $\bm{x}$ in the simplex 
$M(\alpha=1/n)=\{\bm{x} \in \mathbb{R}_{> 0}^n: \sum_{i\in N}x_i= 1\}$, if $x_i$ is the probability of presence of a continuous-time Markovian random walker on node $i\in N$, as transition rates on the network are described by the entries of $-\textbf{L}$, see, e.g., \cite{Lovasz1993} with \cite{Norris1997} chapter 6, and compare to Example \ref{ex:configspace}.
The Markov chain converges to the invariant distribution given by the Perron vector $\bm{q}=\frac{1}{n}\bm{1}$, 
and a natural Lyapunov function in this case is the Kullback-Leibler divergence,
$D(\bm{x}||\frac{1}{n}\bm{1}):=\sum_{i\in N} x_i\log\frac{x_i}{n^{-1}}$, i.e., the relative entropy of the discrete probability distribution $\bm{x}$ w.r.t
 $\frac{1}{n}\bm{1}$.
 Of course, $D(\bm{x}||\frac{1}{n}\bm{1})$ is also a Lyapunov function for average preserving consensus systems, as for instance demonstrated in \cite{Willems1976}, and as the term ``divergence'' suggests, it is also a (collective) disagreement measure.

Just as the sum-of-squares Lyapunov function, also the Kullback-Leibler divergence is additive.
In the following section we first characterize a class of additively structured Lyapunov functions that are strict, i.e., they decrease strictly along trajectories and cease to decrease only at the fixed point.
For this class, we shall solve the Inverse Problem.

\section{Additive Lyapunov functions \& Gradient structures}

\subsection{Lyapunov inequalities:  Convexity \& Relative measures}

 Let $I\subset \mathbb{R}$ be an interval. A function $f: I\to\mathbb{R}$ is said to be convex (concave) if for all $\bm{x} \in I^n$, and positive real numbers $a_1, \ldots,a_n$, such that $\sum_{i=1}^na_i=1$,
\begin{equation}\label{eq:Jensen}
\tag{Jensen}
f\left(\sum_{i=1}^n a_i x_i\right) \overset{(\geq)}{\leq} \sum_{i=1}^n a_i f(x_i).
\end{equation}
Strict inequality yields strict convexity (concavity) of $f$.
\begin{theorem}[Additive Lyapunov functions \& Strict convexity]\label{thm:extensiveLyap}
Consider an irreducible linear consensus system.
Let $H: \mathbb{R} \to \mathbb{R}$ be a continuous function, $\bm{q}$ the left Perron vector, $\beta >0$ and $c>0$  be two parameters.
The additive function
\begin{equation}\label{eq:thm1liapfunc}
V(\bm{x})=\beta \sum_iq_iH(c x_i)
\end{equation} 
 is a strict Lyapunov function if and only if $H$ is strictly convex.
 \end{theorem}
\begin{proof}
A strict Lyapunov function satisfies the
strict inequality $V(\bm{x}(t)) <V(\bm{x}(0))$ along trajectories for any $0<t < \infty$, (except one starts at the equilibrium point).
This is true if $H$ is strictly convex, because then,
\begin{align}
&V(\bm{x}(t))=\beta \sum_{i=1}^n q_iH(c x_i(t))  \\
=&\beta \sum_{i=1}^n q_i H\left(  \sum_{j=1}^n \left[\textbf{P}^{(t)}\right]_{ij} c x_j(0) \right) 
 < \beta \sum_{i=1}^n q_i \sum_{j=1}^n \left[\textbf{P}^{(t)}\right]_{ij}H(c x_j(0) ) \label{eq:liapproofjens1}\\
 = &\beta \sum_{j=1}^n \sum_{i=1}^n  q_i \left[\textbf{P}^{(t)}\right]_{ij} H(c x_j(0)) 
  = \beta \sum_{j=1}^n  q_j H(c x_j(0))= 
 V(\bm{x}(0)).
\end{align} 
Inequality in \eqref{eq:liapproofjens1} follows from \eqref{eq:Jensen}, and $\sum_{i=1}^n  \left[\textbf{P}^{(t)}\right]_{ij} =1$, for any $t> 0$ by construction, see \eqref{eq:perronmatrix} with \eqref{eq:perronmatrix1}.
Conversely, because $V(\bm{x}(0))$ can equivalently be written as the right hand side of the inequality \eqref{eq:liapproofjens1}, and $V(\bm{x}(t))$ as the left hand side of \eqref{eq:liapproofjens1}, the strict inequality \eqref{eq:liapproofjens1} must be true for $V(\cdot)$ to be a strict Lyapunov function.
Suppose $H$ is concave on a small interval $I\subset \mathbb{R}_{> 0} \setminus \{a \bm{1}\}, a \in \mathbb{R}$, choose $\bm{x}(0) \in I^n$, and $t>0$ such that $\bm{x}(t) \in I^n$, as well.
Then, on the small interval $I^n$, the function $V$ increases due to \eqref{eq:Jensen}, which is a contradiction to Lyapunov stability.
Suppose $H$ is linear on $I$, i.e., both convex and concave. Then,
\eqref{eq:liapproofjens1} yields equality by definition of convexity/concavity via \eqref{eq:Jensen}, and $V$ cannot be a strict Lyapunov function.
Thus, $H$ must be strictly convex, which therefore is a necessary and sufficient property .
\end{proof}

If $H$ is strictly concave, then $V(\bm{x})=-\beta \sum_{i=1}^nq_iH(c x_i)$ is a Lyapunov function.

The specific structure that we impose on an additive Lyapunov function with $V_i(x_i)=\beta q_i H(c x_i)$, $i\in N$, implies that, locally, the Lyapunov function may appear distinct in terms of the local constants $q_i$, but a possible non-linearity in the state arising in the specification of the global system quantity $V(\bm{x})$ is identical at all nodes $i\in N$. The positive scalars $\beta$ and $c $ characterize system-wide parameters.

\begin{example}[Electric energy in RC circuits]\label{ex:capacitor} 
Consider a passive LTI electric network of $n$ grounded identical capacitances $C$ whose non-grounded plates are related through a resistance network. The grounded plates are assumed to be at voltage $v_{\textnormal{ref}}$. Each capacitance, of charge $q_i$, injects a current $I_i=-\dot{q}_i$ into a node of the resistance network. Let $\textbf{M}$ be the signed node-to-edge incidence matrix, with every of the $m$ rows containing a $+1$ and a $-1$ encoding the extremities of the edge in the resistance network and the conventional sense of current along that edge. Then the Kirchhoff Current Law, Ohm's Law and Kirchhoff Voltage Law combine classically into $\bm{I}=-\textbf{M}^\top \textbf{W} \textbf{M} \bm{x}$, where $x_i$ is the potential (voltage) at node $i$, and $\textbf{W}$ is the $|E|$-by-$|E|$ diagonal matrix of conductances. The matrix  $\textbf{M}^\top \textbf{W} \textbf{M}$ is none but the symmetric Laplacian matrix $\textbf{L}$ of the resistance network  with conductances as edge weights, called Kirchhoff matrix in this context. From the capacitance's constitutive law, we also know $I_i=C \dot{x}_i$, leading to a global consensus equation, $\dot{\bm{x}}=- C^{-1} \textbf{L} \bm{x}$. The total energy in the circuit is additive and takes the form $\sum_{i=1}^n \frac{1}{2} C\left(x_i-v_{\textnormal{ref}}\right)^2$. The dynamics is therefore a gradient descent of 
the stored electric energy, which satisfies the form  \eqref{eq:thm1liapfunc} with $\beta=nC$, the total system capacitance, $\bm{q}=\frac{1}{n}\bm{1}$, $c=1$, and $H_{\textnormal{elec}}(x)=\frac{1}{2}(x-v_{\mathrm{ref}})^2$.
 The asymptotically reached consensus therefore corresponds to an equalisation of nodal voltage potentials at capacitors, so that no currents flow across the resistors, and stored electric energy reaches a minimum.
 \end{example}


Given an irreducible consensus system with initial condition $\bm{x}_0$ in the positive orthant, one may scale the linear dynamics from state $\bm{x}(t)$ to normalized state $\bm{\rho}(t)=\frac{\bm{x}(t)}{\mathsf{a}}$, that we shall call the \emph{density} of state $\bm{x}(t)$. The density evolves in configuration space $M(\alpha=1)$, with final consensus value $\sum_{i=1}^n  q_i \rho_i(t)=1$, for all times $t \geq 0$.

Define the vector $\bm{p}:=\left( q_1 \rho_1,\cdots , q_n \rho_n \right)^\top$.
Then, $\bm{p}$ is a probability mass vector with components $p_i\in (0,1)$ representing a measure of the probability that a random walker on the graph $\mathsf{G}$, with transition rates given by the entries of $-\textbf{L}$, is present at node $i\in N$, following the equation $\dot{\bm{p}}^{\top}=-\bm{p}^{\top}\textbf{L}$. Those probabilities converge to the stationary probability mass (Perron) vector $\bm{q}$.
The components $\rho_i=p_i/q_i$, represent a probability density function of $\bm{p}$ with respect to the stationary probability of presence $\bm{q}$, providing another intepretation for $\bm{\rho}$. 

The strict Lyapunov function \eqref{eq:thm1liapfunc} in Theorem \ref{thm:extensiveLyap}, if we choose $c=\mathsf{a}^{-1}(\bm{x}(t))$, may be expressed as
\begin{equation}\label{eq:liapprob}
V(\bm{x})=\beta \sum_{i=1}^nq_iH(\rho_i)=\beta \sum_{i=1}^nq_iH\left( \frac{p_i}{q_i}\right).
\end{equation}

\begin{example}[Kullback-Leibler divergence]\label{ex:KL}
The function $H_{\textnormal{Ent}}(x):=x\log x$ is a convex function on $\mathbb{R}_{>0}$. 
For an irreducible consensus system, the choice $\beta=1$ and $c=\mathsf{a}^{-1}(\bm{x}_0)$ yields, with $\bm{\rho}=c\bm{x}$ the strict Lyapunov function
\begin{equation}
V_{\textnormal{KL}}(\bm{x}):= \sum_{i=1}^n q_i \rho_i\log \rho_i 
=\sum_{i=1}^n p_i\log \frac{p_i}{q_i}\triangleq D(\bm{p}||\bm{q}). \label{eq:vkl2}
\end{equation}
This is the Kullback-Leibler divergence introduced in section \ref{ssec:3liap} for  balanced consensus systems, which here is a Lyapunov function for the larger class of irreducible consensus systems.
When the system reaches consensus, then the relative measure is uniform, i.e., $\bm{\rho}=\bm{1}$, as an irreducible Markov chain converges to $\lim_{t\to\infty}\bm{p}(t)=\bm{q}$.
 \end{example}
 
 \begin{example}[Gibbs free energy]\label{ex:gibbs}
In a chemical reaction network the driving forces are the differences of chemical potentials $\mu_i$ of species $i\in N$, (in analogy to voltage potential differences in electric circuits), where
$\mu_i=RT\log \frac{m_i}{m_{i,\mathrm{ref}}}$
with $R$ the gas constant, $T$ the reaction temperature, $m_i$ a concentration given as mole number per reaction volume of species $i$, and $m_{i,\mathrm{ref}}$ a reference value for $m_i$, attained at equilibrium. 
see, e.g., \cite{vdSchaft2013}.
The chemical potentials admit as a Lyapunov function Gibbs free energy, defined as, e.g., in \cite{vdSchaft2013},
 \begin{equation}\label{eq:gibbsa}
V_{\textnormal{Gibbs}}(\bm{m}||\bm{m}_{\mathrm{ref}}):=
RT \sum_{i=1}^n
m_i\log \frac{m_i}{m_{i,\mathrm{ref}}} - (m_i-m_{i,\mathrm{ref}})\geq 0,
\end{equation}
The function \eqref{eq:gibbsa} has a Kullback-Leibler form as \eqref{eq:vkl2} with an added linear term.
In important classes of reaction networks the sum of species concentrations remains constant, cf., e.g., \cite{Kacser1986}, and this
 allows us to bring \eqref{eq:gibbsa} to the form \eqref{eq:thm1liapfunc}:
 Denote the conserved total number of molecules per volume by 
 $\alpha=\sum_{i=1}m_i=constant$, set $\beta=RT \alpha$, $\rho_i=\frac{m_i}{m_{i,\mathrm{ref}}} \triangleq \frac{x_i}{\mathsf{a}(\bm{x})}$, for all $i\in N$, where $\mathsf{a}(\bm{x})=\frac{\alpha}{n}$,
and  $\bm{q}=\frac{1}{\alpha}\bm{m}_{\textnormal{ref}}$, the vector of normalized equilibrium concentrations. Then, \eqref{eq:gibbsa} becomes 
\begin{equation}\label{eq:gibbs}
V_{\textnormal{Gibbs}}(\bm{x})=
RT \alpha \sum_{i=1}^n q_i
\left( \rho_i  \log  \rho_i  -  \rho_i +  1\right)=\beta \sum_{i=1}^nq_iH_{\textnormal{Gibbs}}(\rho_i),
\end{equation}
with $H_{\textnormal{Gibbs}}(\rho_i):=\rho_i(\log \rho_i - 1) + 1 $.

Asymptotically, a reaction network as consensus seeking system
equalizes chemical potential differences, which is achieved when $\bm{m}(t)=\bm{m}_{\mathrm{ref}}$, i.e., when the vector of relative measures $\bm{\rho}$ is uniform, i.e., $\lim_{t\to\infty}\bm{\rho}(t)=\lim_{t\to\infty}\frac{1}{\mathsf{a}}\bm{x}(t)=\bm{1}$, so that Gibbs free energy attains a minimum.
\end{example}
\begin{remark}[Csisz\'{a}r's $f$-divergence]
The class of functions \eqref{eq:liapprob} is known in information theory as $f$-divergences from a probability distribution $\bm{p}$ to another distribution $\bm{q}$, see, e.g., \cite{Csiszar2004}.
\end{remark}

\subsection{Gradient flow in additive potentials }

Let $I\subset \mathbb{R}$ be an interval, $f: I\to \mathbb{R}$ be an increasing continuous function and define the ratio given by divided differences at non-identical points $x,y \in I$,
\begin{equation}\label{eq:Kh}
K_f(x,y):=\frac{x-y}{f(x)-f(y)}.
\end{equation}
\begin{proposition}\label{prop:symm}
The function $K_f(\cdot,\cdot)$ is positive and symmetric in both arguments, i.e., for all non-identical $x,y \in I \subset \mathbb{R}$, $K_f(x,y)=K_f(y,x)>0$. 
\end{proposition}
\begin{proof}
Follows by definition of $K_f(\cdot,\cdot)$ with properties of the function $f$.
\end{proof}

Next, given an undirected $n$-node graph $\mathsf{G}$, with symmetric graph Laplace matrix $\textbf{L}=[l_{ij}]$, a vector $\bm{y} \in \mathbb{R}^n$, and an increasing $\mathscr{C}^1$- function $f:\mathbb{R}\to \mathbb{R}$, we define the matrix $\textbf{K}_f(\bm{y})$ having components
\begin{equation}\label{eq:Kirchhoffmat}
[\textnormal{\textbf{K}}_{f}]_{ij}:=\left\{\begin{tabular}{c l }
$- l_{ij}\cdot  K_{f}\left(y_i,y_j\right)$ & if $(i,j)\in E$,\\
$\sum_{j=1}^n l_{ij} \cdot K_{f}\left(y_i,y_j\right)$, & if $i=j \in N$.
\end{tabular}\right. 
\end{equation}
The matrix $\textbf{K}_f(\bm{y})$ represents a symmetric, state-dependent matrix with  Laplacian form  associated to $\mathsf{G}$, so that it has the same sparsity structure as the usual Laplacian $\textbf{L}$.


With the following theorem, the main result of this paper, we solve the Inverse Problem: Given a symmetric LTI consensus system and an additive Lyapunov function $V(\bm{x})$ as in Theorem \ref{thm:extensiveLyap}, we construct a symmetric matrix function $\textbf{G}^{-1}(\bm{x})$, that has the sparsity structure imposed by $\mathsf{G}$, and
that allows to formulate the LTI consensus dynamics as a Riemannian gradient dynamics driven by $\nabla V(\bm{x})$.

\begin{theorem}[Gradient flow formula]\label{thm:tensorconsensus}
Consider an irreducible symmetric LTI consensus system,
\begin{equation}\label{eq:thmM}
\Sigma: \left\{\begin{tabular}{c}
$\dot{\bm{x}}(t)=-\textnormal{\textbf{L}} \bm{x}(t)$\\
$\bm{x}(0)=\bm{x}_0 \in M$
\end{tabular}\right., \ \ \ M(\alpha)= \left\{\bm{x}\in \mathbb{R}^n_{>0}: \sum_{i=1}^nx_i/n=\alpha>0 \right\}.
\end{equation}
Set the density $\bm{\rho}=\frac{1}{\alpha}\bm{x}$, with consensus value
$\mathsf{a}(\bm{x}_0)=\alpha$.
  Consider a strictly convex $\mathscr{C}^2$-function $H:\mathbb{R}_{>0}\to\mathbb{R}$ so that, according to Theorem \ref{thm:extensiveLyap}, the additive  function
\begin{equation}\label{eq:thmliap}
V(\bm{x})=\sum_{i=1}^nV_i(x_i)=\alpha\sum_{i=1}^n  H(\rho_i),
\end{equation}  
is a strict Lyapunov function for $\Sigma$.
Then, the consensus dynamics generates a gradient flow of $V(\bm{x})$, so that
\begin{equation}\label{eq:consgradK}
\dot{\bm{x}}(t)=-\textnormal{\textbf{L}} \bm{x}(t)=-\textnormal{\textbf{G}}^{-1}(\bm{x}(t))\nabla V(\bm{x}(t)),
\end{equation}
where $\textnormal{\textbf{G}}^{-1}(\bm{x})$ is the matrix 
\begin{equation}\label{eq:Kmat}
 \textnormal{\textbf{G}}^{-1}(\bm{x})=\alpha \cdot \textnormal{\textbf{K}}_{\nabla_{\rho} H}(\bm{\rho})
\end{equation}
with $\textnormal{\textbf{K}}_{\nabla_{\rho} H}(\bm{\rho})$ as defined in \eqref{eq:Kirchhoffmat}.
That is, solutions on $M(\alpha)$ represent a gradient descent flow of $V$ on the set $M(\alpha)$ equipped with a Riemannian structure defined by scalar products via the matrix function $\textnormal{\textbf{G}}^{-1}(\cdot)$.
\end{theorem}

\begin{proof}
By the chain rule,
\begin{equation}\label{eq:VHeq}
\nabla_x V_i(x_i)=\alpha\frac{\partial H(\rho_i)}{\partial \rho_i} \frac{\partial \rho_i}{\partial x_i}=\alpha\,\nabla_{\rho}H(\rho_i)\,\frac{1}{\alpha}=\nabla_{\rho}H(\rho_i).
\end{equation}
Then, for all $i\in N$, 
\begin{align}
  \dot{x}_i
  = &\sum_{j=1}^n  l_{ij}(x_j-x_i)=\sum_{j=1}^n  l_{ij}\,\alpha\frac{1}{\alpha} \frac{\nabla_{\rho} H(\rho_j)-\nabla_{\rho} H(\rho_i)}{\nabla_{\rho} H(\rho_j)-\nabla_{\rho} H(\rho_i)}(x_j-x_i) \\
  =&\sum_{j=1}^n  l_{ij}\,\alpha  \frac{\frac{1}{\alpha}(x_j-x_i)}{\nabla_{\rho} H(\rho_j)-\nabla_{\rho} H(\rho_i)}( \nabla_{\rho} H(\rho_j)-\nabla_{\rho} H(\rho_i))\\  
  \stackrel{\eqref{eq:Kh}}{=}&\sum_i   \alpha\,l_{ij} K_{\nabla_{\rho} H}( \rho_j,\rho_i) ( \nabla_{\rho} H(\rho_j)-\nabla_{\rho} H(\rho_i)) \\
 \stackrel{\text{Prop.}\,\eqref{prop:symm}}{=}&\sum_i  \alpha \, l_{ij} K_{\nabla_{\rho} H}( \rho_i,\rho_j) ( \nabla_{\rho} H(\rho_j)-\nabla_{\rho} H(\rho_i))\\
  \stackrel{\eqref{eq:Kmat},\eqref{eq:VHeq}}{=}&-\sum_{j=1}^n [\textbf{G}^{-1}]_{ij}(x_i,x_j)\nabla_{\rho} H(\rho_j)
  \stackrel{\eqref{eq:thmliap}}{\Leftrightarrow} -\textbf{G}^{-1}(\bm{x})\nabla V(\bm{x})=-\textbf{L}\bm{x}.
\end{align}
The matrix \eqref{eq:Kmat} is symmetric, because $l_{ij}=l_{ji}$ by hypothesis and $K_{\nabla_{\rho} H}(\cdot,\cdot)$ is symmetric in both arguments, see Proposition \ref{prop:symm}.
As $H$ is strictly convex, $\nabla_{\rho} H$ is increasing, so that by Proposition \ref{prop:symm}, the function $K_{\nabla_{\rho} H}(\cdot,\cdot)$ is positive for non-identical arguments.
Due to the Laplacian structure, $\textbf{G}^{-1}(\cdot)$ defines a positive definite bilinear on $T^*_{\bm{x}}M$, the set of Lyapunov function gradients, for all $\bm{x}\in M$, (as required for a gradient flow), because $\mathsf{ker}(\textbf{G}^{-1}(\cdot))=\left\{c\bm{1},c\in \mathbb{R}\right\}$, so that it vanishes only at the equilibrium state.
The limit 
\begin{equation}\label{eq:limK}
\lim_{\rho_j\to \rho_i} K_{\nabla_{\rho} H}(\rho_i,\rho_j)=\left.\left(\frac{\partial^2H(\rho)}{\partial \rho^2}\right)^{-1}\right\vert_{\rho=\rho_i}
\end{equation}
exists, because $H$ is twice differentiable and the second derivative is positive and finite, whenever $\nabla_{\rho} H$ is increasing, so that \eqref{eq:limK} is positive and finite, too.
Eventually, $\textbf{G}^{-1}(\cdot)$ is smoothly varying on $M$, so that \eqref{eq:consgradK} indeed represents Riemannian gradient dynamics.
\end{proof}

The inverse matrix function, $\textbf{G}^{-1}(\cdot)$, appears natural in characterizing gradient structures in consensus systems, rather than $\textbf{G}(\cdot)$ itself, because it inherits the sparsity structure imposed by the interconnection structure $E$ of the graph $\mathsf{G}$.

Notice that strict convexity of $H$ is equivalent to $\partial^2 H(x) /\partial x^2>0$, so that the reciprocal of the scalar curvature of $H$ is also positive. The divided difference components of $\textbf{G}^{-1}$ approximate the reciprocal of the curvature of $H$, compare to \eqref{eq:limK}, so that this choice makes implicitly use of the strict convexity property of the chosen potential, in guaranteeing positive semi-definiteness by construction.

\section{Discussion of non-Euclidean gradient flows}
\subsection{A circuits perspective \&  Disagreement measures \label{ssec:circview}}
The matrix $\textbf{G}^{-1}$, as defined in \eqref{eq:Kmat}, admits a particular factorization:
Let $\textbf{M}$ denote the $|E|\times |N|$ node to edge signed incidence matrix of $\mathsf{G}$, define the edge index $e=(i,j)\in E=\left\{1,2,\ldots,|E|\right\}$, and the $|E|\times |E|$ diagonal matrix having $[\textbf{G}^{-1}]_{ij}$-components on the main diagonal,
\begin{equation}
\textbf{W}_f(\bm{\rho}):=\mathsf{diag}\left( \cdots, \alpha \, l_{ij} \frac{\rho_i-\rho_j}{f(\rho_i)-f(\rho_j)},\cdots \right),
\end{equation}
where we keep $f$ an arbitrary strictly increasing function   for the moment.
Then, we can write $\textbf{G}^{-1}(\bm{x})=\textbf{M}^\top\textbf{W}_f(\bm{\rho})\textbf{M}$, so that the linear consensus gradient dynamics \eqref{eq:consgradK} becomes
\begin{equation}\label{eq:cdynKOhm}
\dot{\bm{x}}=-\textbf{L}\bm{x}=-\textbf{G}^{-1}(\bm{x})\nabla V(\bm{x})=-\textbf{M}^\top\textbf{W}_f(\bm{\rho})\textbf{M}\nabla V(\bm{x}).
\end{equation}

In the context of circuit theory this factorization is known to result from Kirchhoff's current and voltage law together with Ohm's law, where the matrix $\textbf{G}^{-1}$ is also known as Kirchhoff matrix --- a static representation of a resistor network.
In this case, equation \eqref{eq:cdynKOhm} describes the autonomous dynamics of an initially charged passive circuit system, as in Example \ref{ex:capacitor}.

Referring to Figure \ref{fig:KOhmdyn}, the general analogue of Kirchhoff's and Ohm's law find expression in $\textbf{G}^{-1}(\bm{x})=\alpha\textbf{K}_{\nabla_{\rho} H}(\bm{\rho})$, as 
 \begin{equation}\label{eq:KLOhm}
 \bm{u}_N \stackrel{ \text{(KCL)}}{=}\textbf{M}^\top\bm{y}_E, \hspace*{0.5cm} 
 \bm{y}_E\stackrel{ \text{(Ohm)}}{=}\textbf{W}_{\nabla_{\rho}H}(\bm{\rho})\bm{u}_E, \hspace*{0.5cm} 
 \bm{u}_E\stackrel{ \text{(KVL)}}{=}-\textbf{M} \bm{y}_N,
 \end{equation}
 cf., e.g., \cite{Strang2010} ch. 2. 
\begin{figure}[]
\centering
\begin{tikzpicture}[scale=1, circuit ee IEC]
\draw[color=black, rounded corners, fill=gray!30,line width=1pt] (-2.05,0) rectangle (4.15,3);
\draw[color=black,thick, fill=blue!35] (-0.6,1.7) rectangle (2.6,2.8);
 \draw[color=black,thick,fill=red!35] (-1.25,0.2) rectangle (3.4,1.3);
\draw[color=black,thick,fill=white] (-0.25,2) rectangle (0.25,2.6);
\node at (0,2.3) {$\frac{1}{s}$};
\draw[color=black,thick,fill=white]   (1.3,2) rectangle (2.3,2.6);
\node at (1.8,2.3) {$\nabla V(\cdot)$};
\draw[fill]  (-1.8,2.3) circle (0.07);

\node at (-1.25,2.55) {$\bm{u}_N$};
\node at (-0.45,2.55) {$\dot{\bm{x}}$};
\node at (0.9,2.55) {$\bm{x}$};
\node at (3.3,2.55) {$\bm{y}_N$};
\node at (2.15,1.05) {$\bm{u}_E$};
\node at (0,1.05) {$\bm{y}_E$};

\draw[color=black,thick,fill=white] (2.5,1.1) rectangle (3.1,0.5) ;
\node at (2.8,0.8) {$\textbf{M}$};

\draw[color=black,thick,fill=white] (-0.4,1.1) rectangle (-1.05,0.5) ;
\node at (-0.7,0.85) {$\textbf{M}^\top$};

\draw[color=black,thick,fill=white] (0.3,1.1) rectangle (1.75,0.5) ;
\node at (1.05,0.8) {$\textbf{W}_{\nabla_{\rho} H}(\bm{\rho})$};

\draw[-latex] (-1.8,2.3) -- (-0.25,2.3) ;
\draw[-latex] (0.25,2.3) -- (1.3,2.3) ;
\draw[-latex] (2.3,2.3) -- (3.9,2.3)-- (3.9,0.8)-- (3.1,0.8) ;
\draw[-latex] (2.5,0.8) -- (1.75,0.8) ;
\draw[-latex] (0.3,0.8) -- (-0.4,0.8) ;
\draw[-latex](-1.05,0.8)--(-1.8,0.8)--(-1.8,2.25) ;

\node at (-1.6,2.05) {\LARGE-};
\end{tikzpicture}
\caption{Output feedback representation of consensus gradient dynamics \eqref{eq:cdynKOhm} with controller system (red) in Kirchhoff-Ohm factorized form}
\label{fig:KOhmdyn}
\end{figure}
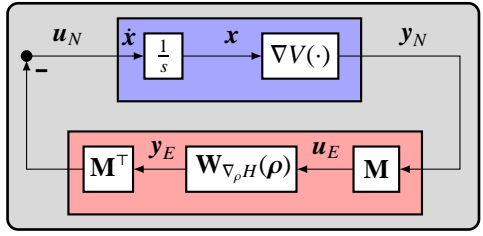
The standard Kirchhoff matrix known from LTI passive circuits is a constant-coefficient graph Laplace matrix. 
This case is recovered by taking a sum-of-squares Lyapunov function as potential, so that the closed loop dynamics is that of a LTI RC circuit: For instance, when we consider $n$ unit-capacitors, the state is a voltage given by $\bm{x}=\bm{v}=\nabla V_{\textnormal{elec}}(\bm{v})=\bm{y}_N$, with
$V_{\textnormal{elec}}(\bm{v})=\frac{1}{2}\sum_{i=1}^n(v_i-v_{\textnormal{ref}})^2$, state velocity corresponds to current, and
$[\textbf{G}_{\textnormal{elec}}^{-1}]_{ij}=l_{ij}\frac{v_i-v_j}{v_i-v_{\textnormal{ref}}-v_j+v_{\textnormal{ref}}}=l_{ij}$ is the conductance of the resistor-edge $(i,j)$, see also Example \ref{ex:capacitor}.

What is new here in our general gradient framework, is that when we do not restrict to sum-of-squares potential functions. 
Ohm's law, as well as the input-output constitutive relation for capacitor nodes, $\bm{y}_N=\nabla V(\frac{1}{s}\bm{u}_N)$, may take non-linear forms, in dependence on the chosen non-quadratic additive $V$. 
Only when we choose the increasing function $f$ in the $K_f(\cdot,\cdot)$-components, that determine the non-linearity in the edge weights, as $f(\cdot)=\nabla_x V_i(\cdot)=\nabla_{\rho}H(\cdot)$ as in \eqref{eq:Kmat},
do the non-linearities in the state to output functions at nodes and in the resistors interconnecting nodes cancel each other.
Otherwise the closed loop dynamics determines a non-linear diffusion:
Set $\nabla_{\rho}H(\rho)=h(\rho)$, and using symmetry of $K_f$ in both arguments, we obtain dynamics
\begin{align}
\dot{x}_i&=\sum_{j=1}^n  \alpha l_{ij} \frac{\rho_j-\rho_i}{f(\rho_j)-f(\rho_i)}\left(h(\rho_j)-h(\rho_i)\right) \\
&=
\sum_{j=1}^n   l_{ij} \frac{h(\rho_i)-h(\rho_j)}{f(\rho_i)-f(\rho_j)}\left(x_j-x_i\right) \ \ \Leftrightarrow \ \
\dot{\bm{x}}=-\textbf{L}_{hf}(\bm{x})\,\bm{x} \label{eq:covar}
\end{align}
as $\frac{h(\rho_j)-h(\rho_i)}{f(\rho_j)-f(\rho_i)}=\frac{h(\rho_i)-h(\rho_j)}{f(\rho_i)-f(\rho_j)}$, for increasing $h$, and with $\textbf{L}_{hf}(\bm{x})$ a state-dependent Laplacian parametrized by increasing functions $f$ and $h$. (An example for this non-linear diffusion dynamics is given in the following section).

We see any additive strict Lyapunov function $V(\bm{x})$ as in Theorem \ref{thm:extensiveLyap} with $c=\mathsf{a}(\bm{x}_0)^{-1}$ as a generalization of the sum-of-squares collective disagreement.
For symmetric consensus systems, we introduce the associated general class of group disagreements as dissipation rate of collective disagreements, generalizing again from the sum-of-squares case:
\begin{align}
\Psi_{\textnormal{V}}(\bm{x}):=-&\dot{V}(\bm{x})=\nabla V(\bm{x})\cdot \textbf{G}^{-1}(\bm{x})\nabla V(\bm{x}) = \nabla V(\bm{x})\cdot \textbf{L}\bm{x} \\
 =&\sum_{(i,j)\in E} l_{ij}\left(x_i-x_j\right)\left(\nabla V_i(x_i)-\nabla V_j(x_j)\right). \label{eq:disagreementgen}
\end{align}
In the RC circuit metaphor, the group disagreement as dissipation rate $\dot{V}_{\textnormal{elec}}(\bm{x})=-\Psi_{\textnormal{L}}(\bm{v})=\bm{v}^\top\textbf{L}\bm{v}$  has the interpretation of the rate of heat transfer from the quadratic electric energy held in states localized at nodes to the environment across edges as resistors.
In analogy to the sum-of-squares energy example
the group disagreements $\Psi_{\textnormal{V}}(\bm{x})$ in \eqref{eq:disagreementgen}, viewed from a circuits perspective, represent a rate of heat production across edges as non-linear resistors.
For symmetric consensus (gradient) systems this group disagreement is a Lyapunov function itself.

In the remaining part we shall detail this non-linear circuits viewpoint using Gibbs free energy as additive potential.

\begin{remark}
The realization of a linear consensus system as output feedback structure as depicted in Figure \ref{fig:KOhmdyn} is a generalization of the known state feedback structure and the linear port-Hamiltonian view on symmetric consensus dynamics \cite{vdSchaft2011a}.
\end{remark}


\subsection{Free energy \& Kullback-Leibler gradient dynamics \label{ssec:relentgrad}  }

Let us apply Theorems \ref{thm:extensiveLyap} and \ref{thm:tensorconsensus} to a symmetric consensus system to obtain a gradient flow of Gibbs free energy as potential function.

\begin{corollary}[Gibbs free energy gradient flow]\label{cor:KLgradflow}
Consider a symmetric consensus system as in Theorem \ref{thm:tensorconsensus} with state density $\bm{\rho}=\frac{1}{\mathsf{a}}\bm{x}$, $\mathsf{a} >0$ the equilibrium value of each $x_i$, $i\in N$, set $\beta=\mathsf{a} n$, and
$H_{\textnormal{Gibbs}}(\rho_i)=\rho_i(\log(\rho_i)-1)$,   $i\in N$. Then,
\begin{equation}\label{eq:Vkla}
V_{\textnormal{Gibbs}}(\bm{x})=\mathsf{a}\sum_{i=1}^n H_{\textnormal{Gibbs}}(\rho_i)=
\sum_{i=1}^nx_i\left(\log \frac{x_i}{\mathsf{a}}-1\right)
\end{equation}
is a strict Lyapunov function and
the consensus dynamics generate a gradient flow
in the additive potential $V_{\textnormal{Gibbs}}(\bm{x})$
with $\textnormal{\textbf{G}}^{-1}_{\textnormal{Gibbs}}$-matrix having non-zero components for $(i,j) \in E$,
\begin{equation}\label{eq:Ggibbs}
[\textnormal{\textbf{G}}^{-1}_{\textnormal{Gibbs}}(\bm{x})]_{ij}:=\mathsf{a}\,l_{ij}K_{\log}(\rho_i,\rho_j)=
\mathsf{a} \, l_{ij}\frac{\rho_i-\rho_j}{\log(\rho_i)-\log(\rho_j)}.
\end{equation}
That is, the shape of the Riemannian manifold is specified by the functions $K_{\log}(\rho_i,\rho_j)$ as defined in \eqref{eq:Kh}, weighted by $\mathsf{a}\,l_{ij}$, the components of the time-invariant graph Laplace matrix scaled by the consensus value.
\end{corollary}
\begin{proof}
The function $V_{\textnormal{Gibbs}}(\bm{x})$ is a strict Lyapunov function by Theorem \ref{thm:extensiveLyap}, with $\bm{q}=\frac{1}{n}\bm{1}$ the Perron vector, and with the choice $c=\frac{1}{\mathsf{a}}$.
For all $i\in N$,  we have
\begin{equation}\label{eq:gibbs1}
\nabla_{\rho} H_{\textnormal{Gibbs}}(\rho_i)= \log( \rho_i ) + \rho_i\frac{1}{\rho_i} -1 =\log\left(\rho_i\right),
\end{equation}
which is an increasing function so that the weights \eqref{eq:gibbs1} are positive definite.
Substitution into the ratio $K_{\nabla_{\rho} H_{\textnormal{Gibbs}}}$ as defined in \eqref{eq:Kh}, yields the $\textbf{G}_{\textnormal{Gibbs}}^{-1}$ components, by definition \eqref{eq:Kmat}
\end{proof}
 
In the circuits metaphor, the gradient of Gibbs free energy as output at a node $i\in N$, $y_{N,i}=\nabla_{\rho} H_{\textnormal{Gibbs}}(\rho_i)=\log \rho_i$,  represent the (non-linear) chemical potential at node $i$, just as a capacitor voltage in the RC-circuit case is a nodal (voltage) potential for sum-of-squares energy, compare to Examples \ref{ex:gibbs} \& \ref{ex:capacitor}. 
When we do not choose the Ohm-type law associated to $y_{N,i}=\nabla_{\rho} H_{\textnormal{Gibbs}}(\rho_i)$ across edges (as assumed in Theorem \ref{thm:tensorconsensus}), but rather the geometry $\textbf{G}^{-1}=\textbf{L}$, then we obtain the non-linear equalizing (diffusion) dynamics for chemical potentials,
\begin{equation}\label{eq:equalgibbs}
\dot{\bm{x}}=-\textbf{L}\nabla V_{\textnormal{Gibbs}}(\bm{x})=:-\textbf{L}_{\log}(\bm{x})\bm{x}.
\end{equation}
with Laplacian $\textbf{L}_{\log}(\bm{x})$ such that $-[\textbf{L}_{\log}(\bm{x})]_{ij}:=\frac{1}{\mathsf{a}}l_{ij}\frac{\log \rho_i - \log \rho_j}{\rho_i-\rho_j}$; here we exploit the Laplacian structure, to obtain the change of variables
 $l_{ij}(\log\rho_j-\log \rho_i)=l_{ij}\frac{\log \rho_i - \log \rho_j}{\rho_i -  \rho_j}(x_j-x_i)\frac{1}{\mathsf{a}}$, as a particular example of \eqref{eq:covar}.

Gibbs free energy and the Kullback-Leibler Lyapunov function
yield the same gradient structure on the set of probability vectors:
 With
$D(\bm{x}||\mathsf{a}\bm{1})=V_{\textnormal{KL}}(\bm{x})=\sum_{i=1}^nx_i\log \frac{x_i}{\mathsf{\mathsf{a}}}$
and $\bm{\rho}=\frac{1}{\mathsf{a}} \bm{x}$, we have $\nabla V_{\textnormal{KL},i}(x_i)=\log\rho_i + 1$, $i\in N$, and therefore, the linear consensus is retrieved as a gradient flow on the Riemannian structure
\begin{equation}\label{eq:lgm}
[\textbf{G}^{-1}_{\textnormal{KL}}(\bm{x})]_{ij}:=\mathsf{a}l_{ij}\frac{\rho_i-\rho_j}{\log\rho_i-\log \rho_j}\triangleq [\textnormal{\textbf{G}}^{-1}_{\textnormal{Gibbs}}(\bm{x})]_{ij},
\end{equation}
following Theorem \ref{thm:tensorconsensus}.

The components $K_{\log}(\cdot,\cdot)$ in \eqref{eq:Ggibbs} and \eqref{eq:lgm} represent the logarithmic mean of its arguments, cf., e.g, \cite{Carlson1972}.
The logarithmic mean yields interesting perspectives to explore convergence bounds for non-linear diffusion such as in \eqref{eq:equalgibbs} in terms of arithmetic and geometric mean inequalities. 

\section{Conclusion }
We presented a generalization of known disagreement measures and gradient structures in linear consensus systems where convexity is instrumental and replaces quadratic measures.
As a particular case, we introduced relative entropy (free energy) as collective disagreement and associated group disagreement.
While non-quadratic quantities are hardly explored in standard linear consensus theory, they are elementary and at the basis of powerful tools in the study of Markov chains, in estimation theory, and statistical physics.
Viewing linear consensus systems from these perspectives offers fruitful research directions.

\section*{Acknowledgment}
The first author would like to thank Dr. Nicolas Hudon for the many discussions
during his research stay at the Universit\'{e} catholique de Louvain.
The second author was supported by the Interuniversity Attraction Pole ``Dynamical Systems, Control and Optimization (DYSCO)'', initiated
by the Belgian State, Prime Minister's Office, and the Action de Recherche Concert\'ee funded by the Federation Wallonia-Brussels.



\section*{Bibliography}
\bibliographystyle{elsarticle-num}
\bibliography{IntEnergyVarConsBib}








\end{document}